\documentclass[draft]{amsart}

\usepackage{amssymb}
\usepackage{url}
\usepackage{amscd}
\usepackage[left=2.5cm,right=2.5cm,top=3.5cm,bottom=3cm]{geometry}

\usepackage[all]{xy}

\def \UU{{\mathcal{U}}}

\def \N{{\mathbb N}}

\def \R{{\mathbb R}}

\def \1{{\mathbb 1}}

\usepackage{amsthm}
\usepackage{ dsfont }
\newtheorem{thm}{Theorem}[section]
\newtheorem{lem}[thm]{Lemma}
\newtheorem{defn}[thm]{Definition}
\newtheorem{definition}[thm]{Definition}
\newtheorem{cor}[thm]{Corollary}
\newtheorem{prop}[thm]{Proposition}

\theoremstyle{definition}
\newtheorem{rem}[thm]{Remark}
\newtheorem{example}[thm]{Example}
\newtheorem{question}[thm]{Open question}

\begin{document}


\title{Compact and Limited operators.}

\author{M. Bachir$^*$, G. Flores$^\dagger$ and S. Tapia-Garc\'ia$^\dagger$}
\date{8th July 2019}

\address{$^*$ Laboratoire SAMM 4543, Universit\'e Paris 1 Panth\'eon-Sorbonne\\ Centre P.M.F. 90 rue Tolbiac\\
75634 Paris cedex 13\\
France}
\email{Mohammed.Bachir@univ-paris1.fr}

\address{$^\dagger$ Departamento de Ingenier\'ia Matem\'atica, CMM (CNRS UMI 2807) Universidad de Chile\\
Beauchef 851\\
Chile}
\email{gflores@dim.uchile.cl, stapia@dim.uchile.cl}



\begin{abstract}
Let $T:Y\to X$ be a bounded linear operator between two normed spaces. We characterize compactness of $T$ in terms of differentiability of the Lipschitz functions defined on $X$ with values in another normed space $Z$. Furthermore, using a similar technique we can also characterize finite rank operators in terms of differentiability of a wider class of functions but still with Lipschitz flavour. As an application we obtain a Banach-Stone-like theorem. On the other hand, we give an extension of a result of Bourgain and Diestel related to limited operators and cosingularity.
\end{abstract}

\maketitle


\newcommand\sfrac[2]{{#1/#2}}

\newcommand\cont{\operatorname{cont}}
\newcommand\diff{\operatorname{diff}}


\section{Introduction}

It is a well-known fact that differentiability in the sense of different bornologies implies distinct properties of the functions depending on the chosen bornology. In this sense, the most common bornologies are those of finite, compact and bounded sets. Each one of them is related to some type of differentiability, more precisely G\^ateaux, Hadamard and Fr\'echet, respectively (see \cite{Ph}). It is shown in \cite{Ba1} that we can characterize limited operators in terms of differentiability of convex functions via the composition with the operator (see Theorem~\ref{limiteddifferentiability} below). To understand the mentioned result, recall the following definitions introduced by Bourgain and Diestel in \cite{BD}.

\begin{defn}
Let $X$ be a Banach space. A subset $A$ of $X$ is called limited if
\[ p_{n} \overset{*}{\rightharpoonup} 0 \implies \lim_{n\to\infty}\sup_{x\in A} |\langle p_{n},x \rangle| = 0, \]
that is, if every weak-$^{*}$ null sequence converges uniformly on $A$.
\end{defn}

\begin{defn}
Let $X$, $Y$ be Banach spaces and $T:Y\to X$ a bounded linear operator. $T$ is called limited if $T(B)$ is limited for every bounded $B\subset Y$. Equivalently, if $\|T^{*}p_{n}\|\overset{n\to\infty}{\longrightarrow} 0$ whenever $p_{n}\overset{*}{\rightharpoonup} 0$.
\end{defn}
We know that every relatively compact subset of $X$ is limited, but the converse is false in general. For useful properties of limited sets and limited operators we refer to \cite{BD}. Considering this, the result in \cite{Ba1} goes as follows.

\begin{thm}\label{limiteddifferentiability}
Let $X$, $Y$ be Banach spaces, $\UU$ be a convex open subset of $X$ and $T:Y\to X$ be a bounded linear operator. Then, $T$ is limited if and only if for every convex continuous function $f:\UU\to \mathbb{R}$, $f\circ T$ is Fr\'echet differentiable at $y\in Y$ whenever $f$ is G\^ateaux differentiable at $Ty\in \UU$. 
\end{thm}

In this sense, a limited operator transforms (for convex functions) G\^ateaux differentiability (the weaker type) into Fr\'echet differentiability (the stronger type) via composition. But considering this, we can go further. In this article we prove (see Theorem \ref{differentiability} and Theorem \ref{compactbump}) an analogue to Theorem \ref{limiteddifferentiability}, by replacing convex continuous functions (resp. limited operator) by Lipschitz functions (resp. compact operators). We also extend this idea to characterize finite-rank operators in the same spirit (see Theorem \ref{finitedifferentiability}). Another way to express these results is the fact that in infinite dimensions, what prevents a continuous convex function $f: X\to \R$ which is G\^ateaux differentiable at some point, from being Fr\'echet differentiable at the same point, is the fact that the identity operator on $X$ is not limited, whereas what prevents a general Lipschitz function which is G\^ateaux differentiable to be Fr\'echet differentiable is the fact that the identity on $X$ is not a compact operator. The non-compactness of the identity operator in infinite dimension is the well known Riesz theorem. On the other hand, the fact that the identity operator is not limited in infinite dimension has been discovered independently by Josefson in \cite{J} and Nissenzweig in \cite{N}.

\vskip5mm

This paper is organized as follows. In Section \ref{S1}, we give some useful notation and definitions. In Section \ref{S2}, we prove our first main result Theorem \ref{differentiability} (and its smooth version Theorem \ref{compactbump}) which gives a connection between compact operators and differentiability of Lipschitz functions. As an application, we obtain a Banach-Stone-type theorem in Theorem \ref{iso}. In Section \ref{S3}, we give our second main result Theorem \ref{BD}, which is an extension of the Bourgain-Diestel result in \cite{BD}. In Section \ref{S4}, we prove our third main result Theorem \ref{finitedifferentiability}, which gives a connection between finite-rank operators and differentiability of finitely Lipschitz functions (see Definition \ref{finitely-Lipschitz}).

\section{Some Notation and Definitions.} \label{S1}

Throughout this paper $X$, $Y$ and $Z$ always denote normed spaces. By $B(x,r)$ and $\overline{B}(x,r)$ we denote the open and closed balls centered at $x$ and of radius $r$, respectively, while $B_X$ and $\overline{B}_X$ denote the open and closed unit ball, respectively, and by $S_X$ its unit sphere. We write explicitly the underlying space if needed. If $A$ is a subset of $X$, by $A^c$ we denote the complement of $A$ in $X$, that is $X\setminus A$. If $\mathcal{V}\subset X$ and $f: \UU\subset X\to Z$, then we denote $f|_{\mathcal{V}}$ as the restriction of $f$ to $\UU\cap \mathcal{V}$. Lipschitz functions and some variations will be often used. Recall the following definitions.

\begin{defn}
Let $X$ and $Z$ be normed spaces. We say that a function $f:U\subset X\to Z$ is locally Lipschitz if for every $u\in U$ there exist constants $r,L>0$ such that $\|f(x)-f(y)\|\leq L\|x-y\|$ for every $x,y\in U\cap B(u,r)$. When there exists a constant $L>0$ such that the last inequality is valid for every $x,y\in U$, we say that the function is Lipschitz.
\end{defn}

We denote by $\textup{Lip}(U,Z)$ the linear space of Lipschitz functions from $U$ to $Z$ and by $\textup{Lip}_{x_{0}}(U,Z)$ the subspace of functions $f\in\textup{Lip}(U,Z)$ which vanish at some fixed point $x_{0}\in U$. The space $\textup{Lip}_{x_{0}}(U,Z)$ endowed with the norm 
\[ \|f\|_{Lip} := \sup_{x,y\in U , x\neq y} \frac{\|f(x)-f(y)\|}{\|x-y\|} \]
is a Banach space. In the case $Z=\mathbb{R}$, we simply write $\textup{Lip}(U)$ and $\textup{Lip}_{x_{0}}(U)$.\\

The set $(C_b(X),\|\cdot\|_{\infty})$ denotes the Banach space of all real-valued bounded continuous functions on $X$, equipped with the sup-norm. By $C_b^G(X)$ we denote the space of all bounded, Lipschitz, G\^ateaux-differentiable functions $f$ from $X$ into $\R$. Let $C_b^F(X)$ be the space of all bounded, Lipschitz, Fr\'echet-differentiable functions $f$ from $X$ into $\R$. These spaces are equipped with the norm $\|f\|=\max(\|f\|_{\infty},\|f'\|_{\infty})$. Recall that by the mean value theorem, we have for every $f\in C_b^G(X)$ that
\[\|f'\|_{\infty}=\sup_{x,\overline{x}\in X; x\neq \overline{x}}\frac{|f(x)-f(\overline{x})|}{\|x-\overline{x}\|}=\|f\|_{Lip}.\]
The spaces $C_b^G(X)$ and $C_b^F(X)$, endowed with the mentioned norm, are Banach spaces (See, \cite{DGZ}).


\section{Differentiability and compact operators} \label{S2}
In this section, we present a result in the line of \cite[Theorem 1]{Ba1}, which is a nice characterization of limited operators in terms of the differentiability of convex functions. In our case, we give a characterization of compact operators in terms of differentiability of Lipschitz functions. The main result of this paper reads as follows.

\begin{thm}\label{differentiability}
Let $X$ and $Y$ be Banach spaces, $\UU\subset X$ an open set and $T:Y\to X$ a bounded linear operator. Then $T$ is compact if and only if for every Banach space $Z$ and every locally Lipschitz function $f:U \to Z$, $f\circ T$ is Fr\'echet differentiable at $y\in Y$ whenever $f$ is G\^ateaux differentiable at $Ty\in \UU$.
\end{thm}

\vskip5mm
The ideas for the proofs of Theorem \ref{differentiability} (and also of Theorem \ref{finitedifferentiability}) are motivated by the following example, in which we find a Lipschitz function $f:\ell^\infty(\N) \to\mathbb{R}$ which is G\^ateaux-differentiable at $0$ and its restriction to $c_0(\N)$ is not Fr\'echet-differentiable at $0$. This type of function cannot exist if the function $f$ is assumed to be convex instead of merely Lipschitz, since the canonical isometry of $c_0(\N)$ into $\ell^\infty(\N)$ is a limited operator (see Theorem \ref{limiteddifferentiability} and \cite{Ba1}).\\
 
Recall that the core of a set $A\subset X$ is the subset $\textup{core}(A)\subset A$ defined by the points $x\in A$ such that for all $y\in X$, there exists $t_y>0$ such that $x+ty\in A$ for all $t\in [0,t_y)$. It is easily seen that $\textup{int}(A)\subset\textup{core}(A)$. Let $f:X\to\R$ be a function. By $\textup{supp}(f)$ we denote the set $\overline{\lbrace x\in X : f(x)\neq 0\rbrace}$.\\

We need the following elementary result.
\begin{prop}\label{intrel} 
Let $f:X\to\R$ be a function such that $0\in \textup{core}(X\setminus \textup{supp}(f))$. Then the function $f$ is G\^ateaux-differentiable at $0$ and $d_Gf(0)=0$.
\end{prop}

\begin{proof}
Take any $h\in X$. Since $0\in\textup{core}(X\setminus \textup{supp}(f))$, there exists $\delta>0$ such that $th\in X\setminus \textup{supp}(f)$ for every $|t|<\delta$, which implies $f(th)=0$ for every $|t|<\delta$. From this we conclude directly that $f$ is G\^ateaux-differentiable at $0$, with $d_{G}f(0)=0$.
\end{proof}

\begin{example}\label{exp}

There exists a Lipschitz function $f:\ell^\infty(\N) \to\mathbb{R}$ which is G\^ateaux-differentiable at $0$ and its restriction to $c_0(\N)$ is not Fr\'echet-differentiable at $0$.

\begin{proof}
Consider the sequence $(y_n)_n\subset \ell^\infty(\N)$ defined by $y_n=\frac{1}{n}e_n$, where $e_n$ is the $n$-th canonical coordinate. Consider the set $C_n=B(y_n,\frac{1}{3n})^c$ and $C=\cap_n C_n$. Consider the function $f:\ell^\infty(\N)\to \mathbb{R}$ defined by
\[f(x)=d(x,C)=\inf \{\|x-z\|:~z\in C\},\]
which is known to be $1$-Lipschitz. Since $0$ belongs to the core of $C$, by Proposition \ref{intrel}, we get that $f$ is G\^ateaux-differentiable at $0$ and its differential is $d_Gf(0)=0$. To show that the restriction of $f$ to $c_{0}$ is not Fr\'echet-differentiable at $0$, it suffices to notice that $(y_{n})_n\subset c_{0}$ and
\[\liminf_{n\to\infty} \dfrac{f(y_n)-f(0)}{\|y_n\|}\geq\liminf_{n\to\infty} \dfrac{1/3n}{1/n}= \dfrac{1}{3}\neq 0.\]
\end{proof}
\end{example}

In the construction of the function in the above example, we used without further detail that the unit vectors of $\ell^{\infty}(\N) $ are ``sufficiently separated" (meaning that they are at distance $1$ to each other) and they belong to the image of the canonical injection $\iota:c_{0}(\N) \to\ell^{\infty}(\N) $, which is not compact. For a proof in a general setting, we need to adapt this idea to obtain a sequence of vectors with the aforementioned property and use them in order to obtain a Lipschitz function that has the required property.

\vskip5mm
Before proceeding with the proof of our result, we need the following useful proposition.

\begin{prop}\label{characterization}
    In the context of Theorem \ref{differentiability}, the following are equivalent
    \begin{enumerate}
        \item For every locally Lipschitz function $f:U\to Z$ it holds that $f\circ T$ is Fr\'echet-differentiable at $y\in Y$ whenever $f$ is G\^ateaux-differentiable at $Ty\in U$.
        \item For every Lipschitz function $f: X\to Z$ it holds that $f\circ T$ is Fr\'echet-differentiable at $y\in Y$ whenever $f$ is G\^ateaux-differentiable at $Ty\in X$.
        \item For every Lipschitz function $f: X\to Z$ it holds that $f\circ T$ is Fr\'echet-differentiable at $0\in Y$ whenever $f$ is G\^ateaux-differentiable at $0\in X$.
        \item For every Lipschitz function $f: X\to \mathbb{R}$ it holds that $f\circ T$ is Fr\'echet-differentiable at $0\in Y$ whenever $f$ is G\^ateaux-differentiable at $0\in X$.
    \end{enumerate}
\end{prop}

\begin{proof}
    It suffices to prove that (4) implies (3) (implications (1)$\implies$(2)$\implies$(3)$\implies$(4) are trivial, while (2)$\implies$(1) and (3)$\implies$(2) are simple by noticing that differentiability notions are local and using translations, respectively).
    
    Suppose that $f:X\to Z$ is a Lipschitz function which is G\^ateaux-differentiable at $0$ and $f\circ T$ is not Fr\'echet-differentiable at $0$. Consider $g:X\to\mathbb{R}$ given by $g(x)=\|f(x)-f(0)-d_{G}f(0)x\|$ (which is trivially Lipschitz). We show that $g$ is G\^ateaux-differentiable at $0$ and $g\circ T$ is not Fr\'echet-differentiable at $0$, which leads to a contradiction. For $h\in S_{X}$ and $t> 0$ we have that
    \[ \frac{g(th)-g(0)}{t} = \frac{\| f(th)-f(0)-td_{G}f(0)h \|}{t} \overset{t\to 0}{\longrightarrow} 0, \]
    which implies that $g$ is G\^ateaux-differentiable at 0, with $d_{G}g(0)=0$. But for $t>0$
    \begin{align*}
        \sup_{h\in S_{X}} \frac{|(g\circ T)(th)-(g\circ T)(0)-t(d_{G}g(0)\circ T)h|}{t} =\\
         \sup_{h\in S_{X}} \frac{\| (f\circ T)(th) - (f\circ T)(0) - t(d_{G}f(0)\circ T)h \|}{t} 
    \end{align*}
    
    which does not converge to $0$ whenever $t$ goes to 0, since $f\circ T$ is not Fr\'echet-differentiable at $0$.

\end{proof}

To get the results, it suffices to work with statement $(4)$ in Proposition \ref{characterization}. 

\subsection{Characterization of compactness}

We begin this subsection with a definition and two results that consist of a central part of the proof of Theorem \ref{differentiability}. Recall that for a convex set $A$, the generated cones $\textup{cone}\{A\}$ is the set $\lbrace \lambda a: \lambda>0, a\in A \rbrace$.

\begin{defn}
Let $(x_n)_n\subset X$ be a bounded sequence. We say that $(x_n)_n$ is $\beta$-separated if $\|x_n\| =\|x_m\|$ for every $n,m\in\mathbb{N}$ and there exists $\beta>0$ such that $\|x_n-x_m\|\geq \beta$ if $n\neq m$. Furthermore, we say that the sequence $(x_n)_n$ is $\beta$-cone-separated if the pairwise intersection of the generated cones $\textup{cone}\{B(x_n,\beta)\}$ is $\{0\}$.
\end{defn}
In the case of $\ell^\infty(\N)$, the canonical sequence of coordinates $(e_n)_n$ is $1$-separated and $(\frac{1}{2}-\varepsilon)$-cone separated, with $\varepsilon\in (0,\frac{1}{2})$, and as the next proposition shows, this is not an special behaviour.

\begin{prop}\label{coneseparated}
Let $(x_n)_n\subset X$ be a sequence which is $\beta$-separated, then it is $\beta/4$-cone-separated.
\end{prop}

\begin{proof}
Let $x\in X\setminus \{0\}$ and $0<\alpha <\|x\|$. We will first show that the set
\[ P_{\alpha}(x) := \left\{ \frac{\|x\|}{\|y\|}y \,:\, y\in B(x,\alpha) \right\} \]
is contained in $B(x,2\alpha)$. Let $y\in B(x,\alpha)$, then:
\[ \left\|x-\frac{\|x\|}{\|y\|}y\right\|\leq \| x-y\| + \left\|y-\frac{\|x\|}{\|y\|} y\right\| \]
\[ = \|x-y\| +|\|y\|-\|x\|| \leq 2\|x-y\| < 2\alpha. \]
Now if $(x_n)_n$ is a $\beta$-separated sequence, consider for each $n\in\mathbb{N}$ the sets $P_{\beta/4}(x_{n})$. With this, $P_{\beta/4}(x_{n})\subset B(x_n,\frac{\beta}{2})$ for each $n\in\mathbb{N}$ which by definition of $\beta$-separated implies that the family $(P_{\beta/4}(x_{n}))_{n}$ is pairwise disjoint. Suppose that there exists $n\neq m$ such that \[ \textup{cone}\{B(x_n,\beta/4)\}\cap \textup{cone}\{B(x_m,\beta/4)\} \neq \{0\} \]
and take $y\neq 0$ belonging to this intersection. By definition, there exist $a, b >0$ and $y_1\in B(x_n,\beta/4)$, $y_2\in B(x_m,\beta/4)$ such that $y = ay_1=b y_2$. Since $\|x_n\|=\|x_m\|$, we see that
\[ \frac{\|x_n\|}{\|y\|}y = \frac{\|x_m\|}{\|y\|}y \]
Equivalently,
\[ \frac{\|x_n\|}{\|y_1\|}y_1 = \frac{\|x_m\|}{\|y_2\|}y_2. \]
This implies that $P_{\beta/4}(x_{n})\cap P_{\beta/4}(x_{m})\neq\emptyset$, which is a contradiction. We conclude that $(x_n)_n$ is $\beta/4$-cone-separated
\end{proof}

\begin{lem}\label{separated}
Let $T:Y\to X$ be a bounded operator which is not compact. Then there exists $\beta>0$ and a $\beta$-separated sequence $(x_{n})_{n}$ in $X$ such that $x_{n}=Ty_{n}$ and $(y_{n})_{n}$ is bounded on $Y$.
\end{lem}


\begin{proof}
Since $TB_Y$ is not a relatively compact subset of $X$, there exists a sequence $(z_n)_n\subset TB_Y$ without cluster points. Without lose of generality, we can assume that for some $\alpha>0$ and for all $n\in\N$, $\|z_n\|\geq\alpha$. Let us define the sequence $(x_n)_n\subset X$ by $x_n:=\alpha z_n/\|z_n\|$, for all $n\in \N$. Notice that $(x_n)_n\subseteq TB_Y$. Since $(z_n)_n$ does not have cluster points, then neither does $(x_n)_n$. With this, since $A=\{x_n:~n\in \N\}\subset TB_Y\cap \alpha S_X$ is a not relatively compact, bounded set, it cannot be totally bounded. Thus, there exists $\beta>0$ such that $A$ cannot be covered by finitely many balls of radius $\beta$. Finally, leaving out eventually some points of $A$, we can construct a $\beta$-separated sequence in $TB_Y\cap \alpha S_X$.
\end{proof}

\begin{rem}\label{remarkseparated}
In the beginning of the proof of Lemma~\ref{separated} we used that $\overline{TB_{Y}}$ is not compact. However, the result is more general. In fact, in the same way we can prove that if $Z$ is an infinite dimensional subspace of $X$, then for some $\beta>0$ there exists a $\beta$-separated sequence in $Z$. In the case that $T$ is compact, we will find a $\beta$-separated sequence in $TY$ if and only if $TY$ is infinite dimensional. However, in this case the associated sequence $(y_{n})_{n}$ cannot be bounded.
\end{rem}
Now we are able to present the proof of the main theorem.

\begin{proof}[Proof of Theorem \ref{differentiability}]

In \cite[Lemma 3.1]{BL}, M. Bachir and G. Lancien proved the necessity, which we include for the sake of completeness. Suppose that $T$ is compact, $Z$ is a Banach space and $f:U\subset X\to Z$ is locally Lipschitz and G\^ateaux differentiable at $Ty$. Since $f$ is locally Lipschitz and G\^ateaux differentiable at $Ty$, we deduce that it is Hadamard differentiable at $Ty$. Then, if $d_{H}f(Ty)$ denotes the Hadamard differential of $f$ at $Ty$ we will have for $t>0$
\[
   \sup_{h\in B_{Y}} \frac{\| (f\circ T)(y+th) - (f\circ T)(y) - t(d_{H}f(Ty)\circ T)h \|}{|t|} \]
\[ = \sup_{k\in \overline{TB_{Y}}} \frac{\| f(Ty+tk) - f(Ty) - t d_{H}f(Ty)k \|}{|t|}.
\]
From this, since $\overline{TB_{Y}}$ is compact, we deduce that if we take limit of $t\to 0$ in the above expresion, the result is $0$. With this, $f\circ T$ is Fr\'echet differentiable at $y$, being its Fr\'echet differential equal to $d_{H}f(Ty)\circ T$.\\

On the other hand, for the sufficiency we proceed by contradiction. Applying Proposition \ref{characterization} we will construct a Lipschitz function $f:X\to\R$. Assume that $T:Y\to X$ is a noncompact bounded operator. Let $(x_n)_n:=(Ty_n)_n\subset TB_Y$ be a $\beta$-separated sequence given by Lemma \ref{separated}, with $(y_n)_n\subset B_Y$. For each $n\in \N$, consider the set $C_n\subset X$ defined by $C_n=B(\frac{x_n}{n},\frac{\beta}{4n})^c=X\setminus B(\frac{x_n}{n},\frac{\beta}{4n})$ and $C:= \cap_n C_n$. By Proposition \ref{coneseparated}, $0$ belongs to the core of $C$ since each line passing through $0$ intersects at most one set of the sets $B(\frac{x_n}{n},\frac{\beta}{4n})$. By Proposition \ref{intrel}, the function $f:X\to\mathbb{R}$ defined by $f(x)=d(x,C)$ (which is $1$-Lipschitz) is G\^ateaux differentiable at $0$ with differential $d_Gf(0)=0$. However, we notice that:
\[ \liminf_{n\to \infty} \frac{(f\circ T)(y_n/n)-(f\circ T)(0)}{\|y_n/n\|} = \liminf_{n\to \infty} \frac{\frac{\beta}{4n}-0}{\|y_n/n\|}\geq\dfrac{\beta}{4\inf_n\{\|y_n\|\}} > 0, \]
which shows that $f\circ T$ is not Fr\'echet differentiable at $0$, since the sequence $(y_n/n)_n$ goes to $0$.
\end{proof}
As consequence, we obtain that in infinite dimentional Banach space $Y$, the set of all Lipschitz continuous functions that vanish at $0$, which are  G\^ateaux differentiable but not Fr\'echet differentiable at $0$, contains a subspace isometric to $\ell^{\infty}(\N)$. More generally, we have the following corollary.
\begin{cor}\label{badfunction} Let $T:Y\to X$ be a non compact bounded operator. Let $F\subset \textup{Lip}_{0}(X)$ be the set defined as follows: $f\in F$ if and only if $f$ is G\^ateaux differentiable at $0$ and $f\circ T$ is not Fr\'echet differentiable at $0$ or $f\equiv 0$. Then, $F$ contains a subspace isometric to $\ell^{\infty}(\N)$.
\end{cor}

\begin{proof}
Consider as before a $\beta$-separated sequence $(x_{n})_n\subset TB_{Y}$ and one of its asociated sequences on $Y$; $(y_{n})_n\subset B_Y$ such that $Ty_{n}=x_{n}$. For every $p\in \mathbb{N}$ prime define the sets $C_{p,n}=B(\frac{x_{p^{n}}}{p^{n}},\frac{\beta}{4p^{n}})^{c}$ and $C_{p}:=\cap_{n}C_{p,n}$. Just as before, the functions $f_{p}(x)=d(x,C_{p})$ are $1$-Lipschitz, G\^ateaux differentiable at $0$ and such that the compositions $f_{p}\circ T$ are not Fr\'echet differentiable at $0$. In the following $(p_{i})_{i}$ stands for an enumeration of the prime numbers.
We have that $(f_{p_{i}})_{i} \subset \textup{Lip}_{0}(X)$ are linearly independent, since their supports are disjoint. Moreover, if $\mu\in\ell^{\infty}(\N)$, the function
\[ f_{\mu}(x): = \sup_{i\in I_+} \mu_{i} f_{p_{i}}(x)-\sup_{i\in I_-}-\mu_{i} f_{p_{i}}(x), \]
where $I_+=\{n\in \N:~\mu_n\geq 0\}$ and $I_-=\N\setminus I_+$, is well defined and $\|\mu\|_{\infty}$-Lipschitz. That is, the operator $L:\ell^{\infty}(\N)\to\textup{Lip}_{0}(X)$ given by $L\mu = f_{\mu}$ is an isometry. By Proposition \ref{coneseparated}, is easy to see that $0\in\textup{core} \left(\cap_{i} C_{p_{i}}\right)$ and by Proposition~\ref{intrel}, $L\mu$ is G\^ateaux differentiable at $0$. But if $\mu\neq 0$, $f_{\mu}$ is not Fr\'echet differentiable at $0$, since if $\mu_{k}\neq 0$, then
\[\liminf_{n\to \infty} \frac{(f_{\mu}\circ T)(y_{p_{k}^{n}}/p_{k}^{n})-(f_{\mu}\circ T)(0)}{\|y_{p_{k}^{n}}/p_{k}^{n}\|} = \liminf_{n\to \infty} \frac{f_{p_{k}}(x_{p_{k}^{n}}/p_{k}^{n})-f_{p_{k}}(0)}{\|y_{p_{k}^{n}}/p_{k}^{n}\|} \]
\[ \geq\liminf_{n\to \infty} \frac{\frac{\beta}{4p_{k}^{n}}-0}{\|y_{p_{k}^{n}}/p_{k}^{n}\|}\geq\frac{\beta}{4\inf_n\{\|y_n\|\}} > 0. \]
\end{proof}

This corollary says that the set $F$ defined there is $c$-lineable, meaning that it contains an isometric copy of a normed space of dimension $c$. More on lineability and spaceability can be found in \cite{ABPS-book}, \cite{ABMP2015}, \cite{GQ2004} and \cite{DF2019}.\\

An interesting case in this framework is whenever the space $X$ admits a smooth bump function. We understand by a bump function a nonnegative continuous function $b:X\to\R$ different from $0$ and with bounded support. The existence of a bump function $b\in C_b^G(X)$ (resp $b\in C_b^F(X)$) is not always trivial and requires in general some geometrical properties on the underlying Banach space $X$. For more information about smooth bump functions in infinite dimensional Banach spaces, we refer to the book of Deville, Godefroy and Zizler \cite{b_DGZ} and the survey \cite{FMc}.  
\begin{thm}\label{compactbump}
Let $X$ and $Y$ be Banach spaces, $U\subset X$ an open set and $T:Y\to X$ a bounded linear operator. Assume that $X$ admits a Lipschitz G\^ateaux differentiable bump function $b:X\to\R$. Then $T$ is compact if and only if for every bounded Lipschitz, everywhere G\^ateaux differentiable function $f:U \to \R$, $f\circ T$ is everywhere Fr\'echet differentiable.
\end{thm}
\begin{proof}
 The necessity is a particular case of Theorem \ref{differentiability}. On the other hand, we assume that $T$ is a noncompact operator and we will construct a bounded Lipschitz and G\^ateaux differentiable function $f:X\to \R$ such that $f\circ T$ fails to be Fr\'echet differentiable. Let $b:X\to \R$ a bump Lipschitz, G\^ateaux differentiable function such that $b(0)> 0$ and $\textup{supp}(b)\subset B_X$. Let $(x_n)_n=(Ty_n)_n\subset TB_Y$ be a $\beta$-separated sequence given by Lemma \ref{separated}, with $(y_n)_n\subset Y$ a bounded sequence. For each $n\in \N$, we define $b_n: X\to \R$ by $b_n(\cdot)= b(n(\cdot-x_n))/n$ and $f= \sum_{i=0}^\infty b_n$. Since the functions $\{b_n:~n\in\N\}$ have pairwise disjoint support and have uniformly bounded Lipschitz constant, $f$ is Lipschitz. Moreover, by construction $f$ is G\^ateaux differentiable on $X$. Now, the proof follows in a similar way as Theorem \ref{differentiability}, leading to that $f$ is not Fr\'echet differentiability at $0$. We leave the details to the reader. 
\end{proof}
\subsection{Application to a Banach Stone Like Theorem}

 The following definition was introduced in \cite{BaJFA} (see also \cite{Ba-S}) and the following axioms in \cite{Ba-S}.
\vskip5mm
\begin{definition} \label{beta} {\bf (The property $P^F$)} Let $(X,d)$ be a complete metric space and $(A,\|\cdot\|)$ be a closed subspace of $C_b(X)$ (the space of all real-valued bounded continuous functions on $X$). We say that $A$ has the property $P^F$ if, for each sequence $(x_n)_n\subset X$, the two following assertions are equivalent:

\begin{enumerate}
    \item The sequence $(x_n)_n$ converges in $(X,d)$.
    \item The associated sequence of the Dirac masses $(\delta_{x_n})_n$ converges in $(A^*,\|\cdot\|_*)$, where the Dirac mass associated to a point $x\in X$, is the linear continuous form $\delta_x: \varphi\mapsto \varphi(x)$ for each $\varphi\in A$.
\end{enumerate}
\end{definition}
\noindent {\bf Axioms.} Let $(X,d)$ be a complete metric space and $A$ be a space of functions included in $C_b(X)$. We say that the space $A$ satisfies the axioms $(A_1)$-$(A^F_4)$ if the space $A$ satisfies:

 $(A_1)$ The space $(A,\|\cdot\|)$ is a Banach space such that $\|\cdot\| \geq \|\cdot\|_{\infty}$,

 $(A_2)$ The space $A$ contain the constants,

 $(A_3)$ For each $n\in \N$ there exists a positive constant $M_n$ such that for each $x\in X$ there exists a function $h_n: X\rightarrow[0,1]$ such that $h_n\in A$, $\|h_n\|\leq M_n$, $h_n(x)=1$ and $\textup{diam}(\textup{supp}(h_n))<\frac{1}{n+1}$. This axiom implies in particular that the space $A$ separate the points of $X$,

$(A^F_4)$ the space $A$ has the property $P^F$.
\vskip5mm
A simple adaptation of the proof in \cite[Proposition 2.5.]{BaJFA}, shows that the spaces $C_b^G(X)$ and $C_b^F(X)$ have the property $P^F$ for every Banach space $X$. In addition, we assume that these spaces contain a bump function respectively, then they will satisfy the axiom $(A_3)$. Thus, the spaces $C_b^G(X)$ and $C_b^F(X)$ satisfy the axioms $(A_1)$-$(A^F_4)$ whenever they contain a bump function respectively, and so we can apply the extension of the Banach-Stone theorem established in \cite[Corollary 1.3.]{Ba-S}. This is what we are going to do in Theorem~\ref{iso}, using our previous Theorem \ref{compactbump}.

\begin{thm} \label{iso} Let $X$ and $Y$ be Banach spaces having a bump function in $C_b^G(X)$ and $C_b^F(Y)$ respectively (see \cite{DGZ}). Then, the following assertions are equivalent.

\begin{enumerate}
    \item There exists an isomorphism $\Phi : C_b^G(X)\to C_b^F(Y)$ such that $\|\Phi(f)\|_{\infty}=\|f\|_{\infty}$ and $\|(\Phi(f))'\|_{\infty}=\|f'\|_{\infty}$ for all $f\in C_b^G(X)$
    \item $X$ and $Y$ are isometrically isomorphic and of finite dimension.
\end{enumerate}

\end{thm}
The proof will be given after the following lemma.

\begin{lem} \label{lem1} For every $a, b\in X$, we have 
\begin{eqnarray*}
\|a-b\|&=&\sup_{f\in C_b^F(X)\setminus \lbrace 0 \rbrace, \|f'\|_\infty > 0} \frac{|f(a)-f(b)|}{\|f'\|_{\infty}}\\
       &=& \sup_{f\in C_b^G(X)\setminus \lbrace 0 \rbrace, \|f'\|_\infty > 0} \frac{|f(a)-f(b)|}{\|f'\|_{\infty}}
\end{eqnarray*}
\end{lem}
\begin{proof} By the Hanh-Banach theorem, there exists $p_{a,b}\in B_{X^*}$ such that be $\|a-b\|=p_{a,b}(a-b)$. For each $\omega >0$, let $\alpha_{\omega}: \R\to \R$ such that $\alpha_{\omega}\in C_b^F(\R)$, $1$-Lipschitz and
\begin{eqnarray*}
\alpha_{\omega} (t)=\left\{ \begin{array}{l} 
t \textnormal{ if } |t|\leq \omega\\
\omega+1 \textnormal{ if } t\geq \omega+2\\
-\omega-1 \textnormal{ if } t\leq -\omega-2
 \end{array} \right.
\end{eqnarray*}
Let us consider the function $f_{\omega}(x)=\alpha_{\omega}\circ p_{a,b}(x)$, for all $x\in X$. We have that $f_{\omega}\in C_b^F(X)$ and is $1$-Lipschitz for every $\omega>0$. By choosing $\omega_0\geq 2\max(\|a\|,\|b\|)$
\begin{eqnarray*}
|f_{\omega_0}(a)-f_{\omega_0}(b)|&=&|\alpha_{\omega_0}\circ p_{a,b}(a)-\alpha_{\omega_0}\circ p_{a,b}(b)|\\
                                 &=&|p_{a,b}(a)-p_{a,b}(b)|\\
                                 &=&|p_{a,b}(a-b)|\\
                                 &=& \|a-b\|.
\end{eqnarray*}
Therefore, $\|f'_{\omega_0}\|_{\infty}=1$. It follows that 
\begin{eqnarray*}
\|a-b\|&\geq&\sup_{f\in C_b^G(X)\setminus \lbrace 0 \rbrace, \|f'\|_\infty > 0} \frac{|f(a)-f(b)|}{\|f'\|_{\infty}}\\
       &\geq& \sup_{f\in C_b^F(X)\setminus \lbrace 0 \rbrace, \|f'\|_\infty > 0} \frac{|f(a)-f(b)|}{\|f'\|_{\infty}}\\
       &\geq& |f_{\omega_0}(a)-f_{\omega_0}(b)|\\
       &=&\|a-b\|
\end{eqnarray*}

\end{proof}

\begin{proof}[Proof of Theorem \ref{iso}] Since $\Phi$ is an isomorphism for the norm $\|\cdot\|_{\infty}$, then from \cite[Corollary 1.3.]{Ba-S}, there exists an homeomorphism $T: Y\to X$ and a continuous function $\epsilon : Y\to \lbrace \pm 1\rbrace$ such that $\Phi(f)(y)=\epsilon(y) f\circ T(y)$ for all $f\in C_b^G(X)$ and all $y\in Y$. Since the space $Y$ is a connected space, we have that $\epsilon=1$ or $\epsilon=-1$. Replacing $\Phi$ by $-\Phi$ if necessary, we can assume without loss of generality that $\Phi(f)=f\circ T$ for all $f\in C_b^G(X)$. We are going to prove that $T$ is an isometry. Let $y_1,y_2\in Y$. Using Lemma \ref{lem1} and the fact that $\|(\Phi(f))'\|_{\infty}=\|f'\|_{\infty}$ for all $f\in C_b^G(X)$, we have 
\begin{eqnarray*}
\|T(y_1)-T(y_2)\|&=&\sup_{f\in C_b^G(X)\setminus \lbrace 0 \rbrace, \|f'\|_\infty > 0} \frac{|f(T(y_1))-f(T(y_1))|}{\|f'\|_{\infty}}\\
                 &=& \sup_{f\in C_b^G(X)\setminus \lbrace 0 \rbrace, \|f'\|_\infty > 0} \frac{|f\circ T(y_1)-f\circ T(y_2)|}{\|f'\|_{\infty}}\\
                 &=& \sup_{f\in C_b^G(X)\setminus \lbrace 0 \rbrace, \|f'\|_\infty > 0} \frac{|\Phi(f)(y_1)-\Phi(f)(y_2)|}{\|(\Phi(f))'\|_{\infty}}\\
                 &=& \sup_{g\in C_b^F(Y)\setminus \lbrace 0 \rbrace, \|g'\|_\infty > 0} \frac{|g(y_1)-g(y_2)|}{\|g'\|_{\infty}}\\
                 &=& \|y_1-y_2\|
\end{eqnarray*}
Thus, $T: Y\to X$ is a surjective isometry. From Mazur-Ulam theorem \cite{JV}, $T$ is an affine map, equivalently $T-T(0)$ is linear. Finally, $T-T(0)$ is a linear surjective isometry from $Y$ onto $X$. So $X$ and $Y$ are isometrically isomorphic. On the other hand, since $f\circ T\in C_b^F(Y)$, whenever $f\in C_b^G(X)$, then $T-T(0)$ is a compact operator by Theorem \ref{compactbump}, which implies thanks to Riesz lemma that $X$ and $Y$ are finite dimensional. Thus, $X$ and $Y$ are finite dimensional and isometrically isomorphic. The converse is clear. Indeed, since G\^ateaux and Fr\'echet differentiability coincides for Lipschitz functions in finite dimensional Banach space, we have that $C_b^G(X)=C_b^F(X)$. On the other hand, if $T: Y\to X$ is an isometric isomorphism, then the operator given by $\Phi(f)=f\circ T$ is an isomorphism between $C_b^F(X)$ and $C_b^F(Y)$ satisfying the two desired conditions. 
\end{proof}
\begin{prop} Let $X$, $Y$ be Banach spaces and $T: Y\to X$ be a bounded linear operator. Then, the following assertions are equivalent. 

\begin{enumerate}
    \item $T$ is compact operator with dense range
    \item The operator $\Phi : C_b^G(X)\to C_b^F(Y)$ defined by $\Phi(f)=f\circ T$ is a well-defined injective bounded linear operator.
\end{enumerate}

\end{prop}
\begin{proof} Suppose that $T$ is compact, then $f\circ T \in C_b^F(Y)$ whenever $f\in C_b^G(X)$ by \cite[Lemma 3.1.]{BL}, so $\Phi$ maps $C_b^G(X)$ into $C_b^F(Y)$. By the density of the range of $T$, $\Phi$ is injective. Then, it is clear that $\Phi$ is a bounded linear operator satisfying $\|\Phi(f)\|_{\infty}\leq\|f\|_{\infty}$ and $\|(\Phi(f))'\|_{\infty}\leq\|f'\|_{\infty}\|T\|$ for all $f\in C_b^G(X)$. For the converse, since $\Phi$ maps $C_b^G(X)$ into $C_b^F(Y)$, then by Theorem \ref{compactbump} the operator $T$ is compact. Suppose by contradiction that $\overline{T(Y)}\neq X$. There exists $x_0\in X$ such that $x_0\not \in \overline{T(Y)}$. By the Hanh-Banach theorem, there exists a linear continuous form $p\in X^*$ such that $p(x_0)=1$ and $p_{\overline{T(Y)}}=0$. Let $\alpha : \R\to \R$ be such that $\alpha\in C_b^F(\R)$ and
\begin{eqnarray*}
\alpha (t)=\left\{ \begin{array}{l} 
2t-1 \textnormal{ if } 1\leq t\leq 2\\
4 \textnormal{ if } t\geq 4\\
0 \textnormal{ if } t\leq 0
 \end{array} \right.
\end{eqnarray*}
Let us define $f_0(x)=\alpha\circ p$. Thus, $f_0\in C_b^G(X)$ and we have $f_0\circ T =0$. Thus, $\Phi(f_0)=\Phi(0)=0$ but $f_0\neq 0$ since $f_0(x_0)=1$. This contradict the injectivity of $\Phi$.
\end{proof}


\section{Limited operators and co-Lipschitz delta-convex mappings} \label{S3}

We recall the following definition introduced in \cite{BJLPS} and studied in several papers (see for instance \cite{M} and \cite{R}).
\begin{defn} Let $X$ and $E$ be two metric spaces and $f : X \to E$ be a mapping. 
We say that $f$ is co-Lipschitz, if there exists a constant $c>0$ such that for all $x\in X$ and all $r>0$ we have
\begin{eqnarray*}
B(f(x), cr) \subset f(B(x,r)).
\end{eqnarray*}
\end{defn}
The co-Lipschitz property is intimately related with the metric regularity of functions, since if a function $f:X\to E$ is co-Lipschitz, then is metrically regular near $(\overline{x},f(\overline{x}))$ for every $\overline{x}\in X$. The definition (and more) of metric regularity of multivalued functions can be found in \cite{IOFFE}, \cite{RW} and references therein. We also recall the definition of delta-convex mapping introduced by Vesel\'{y} and Zaj\'{i}\v{c}ek in \cite{VZ}.
\begin{defn} Let $X$ and $E$ be two Banach spaces and $h : X \to E$ be a map. 
We say that $h$ is d.c mapping, that is delta-convex mapping if and only if, there exists a convex continuous function, called a control function, $g: X\to \R$, such that $e^*\circ h +g$ is convex continuous for every $e^*\in E^*$. 
\end{defn}
\begin{thm} \label{dc} (see \cite[Theorem 4.]{DVZ}) Let $X$ and $E$ be two Banach spaces and $h : X \to E$ be a d.c mapping with a control function $g$. If $g$ is Fr\'echet-differentiable at $x$ then, $h$ is Fr\'echet-differentiable at $x$.
\end{thm}
\paragraph{\bf Example of Fr\'echet-differentiable not linear d.c mapping.} Every linear bounded mapping is trivially a Fr\'echet-differentiable d.c mapping, but the class of d.c mappings with Fr\'echet-differentiable control function is larger. We know, for example, from \cite[Proposition 1.11.]{VZ} that if $H$ is a Hilbert space and $Y$ is any normed space, then every $C^{1,1}$ mapping $h: H\to Y$ is d.c with a control function $\textnormal{Lip}(h')\|\cdot\|^2$ (which is an everywhere Fr\'echet-differentiable control function). For interesting results about d.c Lipschitz isomorphic normed spaces, we refer to \cite{DVZ}. We need the following proposition.
\begin{prop} \label{conv} Let $X, Y$ be a Banach spaces and $h : X \to Y$ be a continuous d.c mapping with a control function $g$. Let $f : Y \to \R$ be a convex continuous function. Then, $f\circ h +g$ is convex continuous.

\end{prop}
\begin{proof} Since $h$ is d.c mapping , there exists a convex continuous function $g: X\longrightarrow \R$, such that $e^*\circ h +g$ is convex continuous for every $e^*\in Y^*$. By the Fenchel theorem, there exists a family of continuous linear maps $(e^*_i)_{i\in I}$ and real numbers $(c_i)_{i\in I}$, such that $f(y)=\sup_{i\in I} e^*_i (y)+c_i$ for all $y\in Y$. Thus, $f(h(x)) +g(x)= \sup_{i\in I} \lbrace e^*_i \circ h(x)+c_i+g(x)\rbrace$ for all $x\in X$. It follows that $f\circ h + g$ is convex, since $e^*_i \circ h +c_i +g$ is convex for each $i\in I$. On the other hand, since $f, h$ and $g$ are continuous, then $f\circ h +g$ is continuous.
\end{proof}

We recall that the Bourgain-Diestel theorem established in \cite{BD} says that: if a bounded linear operator $T : Y\to X$ is limited then it is strictly cosingular, that is, the only Banach spaces $E$ for which we can find a linear bounded operator $q_X : X \to E$ such that $q_X\circ T$ is surjective, are finite dimensional. This result applied with the identity map, gives the well known Josefson-Nissenzweig theorem \cite{J}, \cite{N}, with another proof. We give below an extension of the Bourgain-Diestel theorem, where the assumption of bounded linearity of $q_X$ is replaced by the more general condition of co-Lipschitz d.c mapping. Our proof is based on Theorem \ref{limiteddifferentiability} and the existence of convex G\^ateaux not Fr\'echet differentiable function at some point in infinite dimensional, which is a result based on the Josefson-Nissenzweig theorem. For a canonical contruction of such functions, we refeer to \cite{Ba1}. Thus, Theorem \ref{limiteddifferentiability} together with Josefson-Nissenzweig theorem implies the following extension of Bourgain-Diestel theorem.
\begin{thm} \label{BD} ({\bf Extension of Bourgain-Diestel theorem }) Let $T : Y\to X$ be a bounded linear operator. Then, we have that $(1) \Longrightarrow(2)\Longrightarrow (3)\Longrightarrow (4)$.

\begin{enumerate}
    \item $T$ is limited operator.
    \item The only Banach spaces $E$ for which we can find a d.c mapping $q_X : X \to E$ with a control function $g$ Fr\'echet-differentiable at $0$ and such that $q_X (0)=0$ and $B_E(0,t)\subset q_X\circ T(B_Y(0,t))$ for all $t>0$, are finite dimensional.
    \item The only Banach spaces $E$ for which we can find a d.c mapping $q_X : X \to E$ with a control function $g$ Fr\'echet-differentiable at some point $x_0$ and such that $q_X\circ T$ is co-Lipschitz, are finite dimensional.
    \item The only Banach spaces $E$ for which we can find a linear bounded operator $q_X : X \to E$ such that $q_X\circ T$ is surjective, are finite dimensional.
\end{enumerate}

\end{thm}
\begin{proof} The implications $(2)\Longrightarrow (3)\Longrightarrow (4)$ are trivial. Let us prove $(1)\Longrightarrow (2)$. In fact, we are going to prove $\neg (2)\Longrightarrow \neg (1)$. Indeed, let $E$ be an infinite dimensional Banach space. We can find a convex Lipschitz continuous function $f : E \to \R$ which is G\^ateaux-differentiable but not Fr\'echet-differentiable at $0$, with $f(0)=0$ and $f'(0)=0$ (see a construction of such a function in \cite{Ba1}). Since $f$ is not Fr\'echet-differentiable at $0$, then 
\begin{eqnarray*}
(\exists\, \varepsilon >0) (\forall\delta >0) (\forall t, 0<t\leq \delta) \sup_{h\in B_E(0,1)} t^{-1} |f(th)|\geq \varepsilon.
\end{eqnarray*}
Suppose that $q_X : X \to E$ is a d.c mapping with a control function $g$ Fr\'echet-differentiable at $0$, such that $q_X(0)=0$ and $B_E(0,t)\subset q_X\circ T(B_Y(0,t))$, for all $t>0$. Since $f'(0)=0$, the only candidate for Fr\'echet-derivative of $f\circ q_X\circ T$ at $0$ is $0$. But since $q_X(0)=0$ and $B_E(0,t)\subset q_X\circ T(B_Y(0,t))$ for all $t>0$, we have for every $0< t \leq \delta$,
\begin{eqnarray*}
\sup_{h\in B_Y(0,1)} t^{-1} |f\circ q_X\circ T(th)| \geq \sup_{h\in B_E(0,1)} t^{-1} |f(th)|\geq \varepsilon.
\end{eqnarray*}
This shows that $f\circ q_X\circ T$ is not Fr\'echet-differentiable at $0$. By Proposition \ref{conv}, $f\circ q_X +g$ is convex continuous and $g$ is Fr\'echet-differentiable at $0$. Using Theorem \ref{dc}, $q_X$ is Fr\'echet-differentiable at $0$, since its control function $g$ has that property. It follows that $f\circ q_X+ g$ is a convex continuous function G\^ateaux-differentiable at $0=T(0)$. Suppose by contradiction that $T$ is limited. Then, $f\circ q_X\circ T +g\circ T$ is Fr\'echet-differentiable at $0$ by \cite[Theorem 1.]{Ba1} which implies that $f\circ q_X\circ T$ is Fr\'echet-differentiable at $0$ (since $g\circ T$ has that property) which is a contradiction. Hence, $T$ is not limited. 
\end{proof}
\begin{cor} Let $E$ be any infinite dimensional Banach space. Then, there exists no d.c mapping $q :\ell^{\infty}(\N) \to E$ with a control function $g$ Fr\'echet-differentiable at $0$ such that the restriction $q_{|c_0(\N)}$ (to $c_0(\N)$) is co-Lipschitz.
\end{cor}
\begin{proof} The proof follows from Theorem \ref{BD}, since the canonical embedding from $c_0(\N)$ into $\ell^{\infty}(\N)$ is a limited operator (see \cite{BD}).

\end{proof}
\begin{rem} Every convex Lipschitz continuous function from $f: c_0(\N)\to \R$ has a convex Lipschitz continuous extension to $\ell^{\infty}(\N)$, it suffices to consider the function 
$$\tilde{f}(x):=\inf_{y\in c_0(\N)} \lbrace f(y) +L(f)\|x-y\|_{\infty}\rbrace,$$
where $L(f)$ denotes the Lipschitz constant of $f$. However, if a function $f:c_0(\N)\to \R$ is G\^ateaux-differentiable but not Fr\'echet-differentiable at some point $a\in c_0(\N)$, then $f$ cannot have an extension to $\ell^{\infty}(\N)$ which is convex Lipschitz continuous and G\^ateaux-differentiable at $a$. This is due to the fact that the identity mapping $i :c_0(\N)\longrightarrow \ell^{\infty}(\N)$ is a limited operator, and so by \cite{Ba1} every restriction to $c_0(\N)$ of convex Lipschitz continuous and G\^ateaux-differentiable function at $a$, must be Fr\'echet-differentiable at $a$.
\end{rem}


\section{Finite Rank Operators} \label{S4}

Our final aim is to shown how the idea of the proof of Theorem \ref{differentiability} can be used to characterize finite rank operators. The forthcoming result involves the bornology generated by the bounded sets of $X$ which are contained in finite dimensional subspaces. To get this result, we introduce a new class of functions which contains all the Lipschitz functions.

\begin{defn} \label{finitely-Lipschitz}
We say that a function $f:\UU\subset X\to Z$ is finitely locally Lipschitz if for every finite dimensional affine subspace $Y$ of $X$ such that $\UU\cap Y\neq\emptyset$, $f|_{(Y\cap \UU)}$ is locally Lipschitz. If every restriction is Lipschitz, we simply say that $f$ is finitely Lipschitz. We denote by $\textup{FLip}(\UU,Z)$ the linear space of finitely Lipschitz functions from $\UU\subset X$ to $Z$ and by $\textup{FLip}_{x_0}(\UU,Z)$ the subspace of functions $f\in\textup{FLip}(X,Z)$ which vanish at some fixed point $x_0\in U$. In the case $Z=\mathbb{R}$, we simply write $\textup{FLip}(\UU)$ and $\textup{FLip}_{x_{0}}(\UU)$.
\end{defn}


\begin{rem}
Let $X$ and $Z$ be two Banach spaces, $\UU\subset X$ be any open set and $x_{0}\in \UU$. In case that $X$ is finite dimensional, it is clear that $\textup{FLip}(\UU)=\textup{Lip}(\UU)$ (the same holds for the local concept), but in general the inclusion $\textup{Lip}(\UU)\subset\textup{FLip}(\UU)$ is strict (the same holds for the local concept). In fact, whenever $\UU=X$, the equality holds if and only if $X$ is finite dimensional. This is easily seen by noticing that if $\varphi$ is any not bounded linear functional, then $\varphi$ belongs to $\textup{FLip}(X)\setminus \textup{Lip}(X)$. As we can see, in infinite dimension a finitely locally Lipschitz functions does not even need to be continuous. 
\end{rem}

\begin{thm}\label{finitedifferentiability}
Let $X$ and $Y$ be Banach spaces, $\UU\subset X$ be an open set and $T:Y\to X$ be a bounded linear operator. Then $T$ has finite rank if and only if for every Banach space $Z$ and every finitely locally Lipschitz function $f:\UU\to Z$, $f\circ T$ is Fr\'echet differentiable at $y\in Y$ whenever $f$ is G\^ateaux differentiable at $Ty\in \UU$.
\end{thm}

\begin{rem}
By similar arguments, we deduce a simplification of the statement of Theorem \ref{finitedifferentiability} as we did in Theorem \ref{differentiability} (using Proposition \ref{characterization}). This means that for the sufficiency, we will show the proof for $\UU=X$, $Z=\mathbb{R}$ and $y=0$.
\end{rem}

\begin{proof}
The necessity part goes alongs the lines of the necessity of Theorem \ref{differentiability}. Let $T$ be a finite rank operator and let $f:\UU\subset X\to Z$ be a locally finitely Lipschitz and G\^ateaux differentiable at $Ty\in \UU$. Since $TY$ is a finite dimensional subspace of $X$, the function $g:=f|_{TY}$ is Lipschitz and Fr\'echet differentiable at $Ty$. Then, if $d_Fg(Ty)$ denotes the Fr\'echet differential of $g$ at $Ty$, we have that
\[
   \sup_{h\in B_{Y}} \frac{\| (f\circ T)(y+th) - (f\circ T)(y) - t(d_{F}g(Ty)\circ T)(h) \|}{|t|} \]
\[= \sup_{u\in \overline{TB_{Y}}} \frac{\| f(Ty+tu) - f(Ty) - t d_{F}g(Ty)(u) \|}{|t|}
\]
\[= \sup_{u\in \overline{TB_{Y}}} \frac{\| g(Ty+tu) - g(Ty) - t d_{F}g(Ty)(u) \|}{|t|}
\]
From the last line, since $g$ is Fr\'echet differentiable at $Ty$, we deduce that the first supremum goes to $0$ when $t\to 0$. Then $f\circ T$ is Fr\'echet differentiable at $y$, being its Fr\'echet differential equal to $d_{G}g(Ty)\circ T$.\\

To prove the sufficiency we will proceed by contradiction. Suppose that $T:Y\to X$ is a bounded operator such that $TY$ is infinite dimensional. By Remark \ref{remarkseparated} we have that there exists a $\beta$-separated sequence $(x_n)_n$ in $TY$. Since $x_n\in TY$, for each $n$ there exists $y_n\in Y$ such that $Ty_n= x_n$. Now, define the sequence of subsets 
\[ C_n= B\left(\frac{x_n}{n\|y_n\|}, \frac{\beta}{4n\|y_n\|}\right)^{c}.\]
Since $(x_n)_n$ is a $\beta$-separated sequence, by Proposition \ref{coneseparated} we deduce that the family $(C_n^c)_n$ is pairwise disjoint. With this, the functions $f_n:X\to\mathbb{R}$ given by
\[ f_n(x)=\|y_n\|d(x,C_n)\quad n\in\mathbb{N} \]
are $\|y_n\|$-Lipschitz and have pairwise disjoint support. Consider now $f:X\to\mathbb{R}$ given by
\[ f(x) = \sup_{n} f_{n}(x). \]
It is easy to see that this function is well defined. \\

We claim that $f\in\textup{FLip}(X)$. Let $V$ be a finite dimensional affine subspace of $X$ and suppose that $V$ intersects infinitely many of the sets $(C_n^c)_n$, namely $(C_{n_k}^c)_k$. Take any sequence $(v_k)_k$ such that $v_k\in V\cap C_{n_k}^c$. If we consider $v'_k:=n_k\|y_k\|v_k$, we see that for $k,j\in\mathbb{N}$
\[ \beta \leq \| x_{n_k} - x_{n_j} \| \leq \| x_{n_k} - v'_k \| + \| v'_k - v'_j \| + \| v'_j - x_{n_j} \| \leq \frac{\beta}{2} + \| v'_k - v'_j \|. \]
which implies that $(v'_k)_k\subset V$ does not have accumulation points, which is impossible because $V$ is finite dimensional and
\[ \|v'_k\| \leq \|v'_k-x_{n_k}\| + \|x_{n_k}\| \leq \frac{\beta}{4} + M, \]
where $M=\|x_n\|$ (for all $n$). Then, $V$ intersects only finitely many of the sets $(C_n^c)_n$, namely $(C_{n_k}^c)_{k=1}^N$. Since the functions $f_n$ have pairwise disjoint support it follows that $\textup{Lip}(f|_V) \leq \max\{\|y_{n_k}\|\,:\,k=1,...,N\}$, which proves the claim. Similar to the Theorem \ref{differentiability}, we have that $0$ belongs to the core of the complement of the support of $f$. Thus, by Proposition \ref{intrel} we deduce that $f$ is G\^ateaux differentiable at $0$, and $d_Gf(0)=0$. However, we notice that:
\[\liminf_{n\to \infty} \frac{f\circ T (y_n/n\|y_n\|)-f\circ T(0)}{\|y_n/n\|y_n\|\|} = \liminf_{n\to \infty} \frac{ \|y_n\|\frac{\beta}{4n\|y_n\|}-0}{\frac{1}{n}} = \frac{\beta}{4} > 0, \]
which shows that $f\circ T$ is not Fr\'echet differentiable at $0$ since the sequence $(y_n/n\|y_n\|)_n$ goes to $0$.
\end{proof}

\begin{cor}\label{badfunction2} Let $T:Y\to X$ be a bounded operator with infinite rank. Then the set $F\subset \textup{FLip}(X)$ such that $f\in F$ if and only if $f$ is G\^ateaux differentiable at $0$ and $f\circ T$ is not Fr\'echet differentiable at $0$ or $f\equiv 0$, contains a subspace algebraically isomorphic to $\ell^{\infty}(\N)$.
\end{cor}

\begin{proof}
This proof goes similar to Corollary \ref{badfunction} but with suitable modifications. Consider as before a $\beta$-separated sequence $(x_n)_n:=(Ty_n)_n\subset TY$. Let $p\in \mathbb{N}$ be a prime number and, for $n\in \mathbb{N}$, define the set $C_{p,n}$ and the functions $f_{p^n},~g_p:X\to\R$ by

\[ C_{p,n}=B\left(\dfrac{x_{p^n}}{n\|y_{p^{n}}\|},\dfrac{\beta}{4n\|y_{p^{n}}\|}\right)^{c},~ f_{p^n}(x)= \|y_n\|\textup{dist}(x, C_{p,n}),\]

\[~ \text{and }g_p(x)=\sup_{n\in \N} f_{p^n}(x). \]

Thanks to Theorem \ref{finitedifferentiability}, we know that $g_p$ belongs to $\textup{FLip}(X)$, is G\^ateaux differentiable at $0$ and $g_p\circ T$ is not Fr\'echet differentiable at $0$. Let $(p_n)_n$ be an enumeration of prime numbers. Since $(x_n)_n$ is $\beta$-separated, the sets $\{\textup{supp}(g_{p_n})_n:~n\in \N\}$ are pairwise disjoint. In a similar way as in the proof of Corollary \ref{badfunction}, we can deduce that the linear operator $L:\ell^\infty(\N)\to \textup{FLip}(X)$ given by 
\[ L\mu(x): = \sup_{i\in I_+} \mu_{i} g_{p_{i}}(x)-\sup_{i\in I_-}-\mu_{i} g_{p_{i}}(x), \]
where $I_+=\{n\in \N:~\mu_n\geq 0\}$ and $I_-=\N\setminus I_+$, is well defined, injective and satisfies that for every $\mu\in \ell^\infty(\N)$ $\mu\neq 0$, $L\mu$ is G\^ateaux differentiable at $0$, but $L\mu \circ T$ is not Fr\'echet differentiable at $0$.
\end{proof}

\begin{rem}
By construction, none of the functions $f\in F$ evoked in Corollary \ref{badfunction2} are Lispchitz, but instead, they are finitely Lipschitz. Thus, if $T:Y \to X$ is a noncompact operator, then we can find two subspaces of $\textup{FLip}(X)$, infinite dimensionals with uncountable Hammel basis, $F_1$ and $F_2$ (given by Corollary \ref{badfunction} and Corollary \ref{badfunction2} respectively) such that satisfy for each $f\in F_1\cup F_2$, $f$ is G\^ateaux differentiable at 0, $f\circ T$ is not Fr\'echet differentiable at 0. Since we can ensure that in $F_2$ there are no Lipschitz function except from $0$, $F_1\cap F_2=\{0\}$.
\end{rem}

\begin{question}
In view of \cite[Theorem 1]{Ba1}, Theorem \ref{differentiability} and Theorem \ref{finitedifferentiability}, it seems interesting to study the case of weak-Hadamard differentiability, and furthermore the abstract framework of $\beta$-differentiability, with $\beta$ a given bornology on $X$. That is, since for a bornology $\beta$ on $X$ we can define a notion of differentiability: Is it true that there exists a class of functions $\mathcal{F}_\beta$ such that an operator $T:Y\to X$ sends the unit ball $B_Y$ to $TB_Y\in\beta$ iff for each function $f\in \mathcal{F}_\beta$, $f\circ T$ is Fr\'echet differentiable at $y\in Y$, whenever $f$ is G\^ateaux differentiable at $Ty$? In \cite{Ba1} is shown that this is true whenever $\beta$ is the bornology of limited sets, by using the convex functions. We have shown the same result but for the Hadamard bornology by using the Lipschitz functions and moreover, we have defined the finitely Lipschitz functions for the bornology generated by bounded convex sets with finite dimensional span, which more or less corresponds to the convex version of the G\^ateaux bornology.
\end{question}

\section*{Acknowledgment}
The authors would like to thank A. Daniilidis for fruitful discussions. The second author was supported by the grants CONICYT-PFCHA/Doctorado Nacional/2017-21170003, FONDECYT 1171854 and CMM Grant AFB170001. The third author was supported by the grants CONICYT-PFCHA/Doctorado Nacional/2018-21181905, Monge Invitation Programe of \'Ecole Polytechnique, FONDECYT 1171854 and CMM Grant AFB170001.

\bibliographystyle{amsplain}

\end{document}